\documentclass[12pt]{amsart}

\font\bbbld=msbm10 scaled\magstephalf
\newcommand{\ve}{{\bf e}}

\newcommand{\bM}{\bar{M}}

\newcommand{\bfR}{\hbox{\bbbld R}}
\newcommand{\bfS}{\hbox{\bbbld S}}

\newcommand{\cC}{\mathcal{C}}

\newcommand{\tr}{\mbox{tr}}

\newcommand{\tF}{\tilde{F}}

\newcommand{\tc}{\tilde{c}}
\newcommand{\tm}{\tilde{m}}
\newcommand{\tu}{\tilde{u}}

\newcommand{\ol}{\overline}
\newcommand{\ul}{\underline}

\newtheorem{theorem}{Theorem}[section]
\newtheorem{lemma}[theorem]{Lemma}

 \theoremstyle{definition}

\theoremstyle{remark}
\newtheorem{remark}[theorem]{Remark}

\numberwithin{equation}{section}

\begin{document}

\title[fully nonlinear elliptic equations]
{The Dirichlet Problem for Fully Nonlinear Elliptic Equations
on Riemannian Manifolds
%in Bounded Domains
}
\author{Bo Guan}
\address{Department of Mathematics, Ohio State University,
               Columbus, OH 43210, USA} 
\email{guan@math.osu.edu}
%\address{School of Mathematical Sciences, Xiamen University, %Xiemen, China}

\thanks{Research of the author was supported in part by NSF grants.}
%\date{}

\begin{abstract}
We solve the Dirichlet problem for fully nonlinear elliptic equations on Riemannian manifolds under essentially optimal structure conditions, especially with no restrictions to the curvature of the underlying manifold
and the second fundamental form of its boundary. 
The main result (Theorem~\ref{3I-thm1})
includes a new (and optimal) result in the Euclidean case.
We introduce some new ideas and methods in deriving {\em a priori} 
estimates, which can be used to treat other types of
fully nonlinear elliptic and parabolic equations on real or complex manifolds.

\end{abstract}

\maketitle

\bigskip

\section{Introduction}
\label{3I-I}

\medskip

Let $(\bM^n, g)$ be a compact Riemannian manifold of dimension $n \geq 2$
with smooth boundary $\partial M$, and $M$ denote the interior of $\bM$ so $\bM = M \cup \partial M$.
In this paper we are concerned with 
fully nonlinear elliptic equations of the form
\begin{equation}
\label{3I-10}
%\left\{ \begin{aligned}
f (\lambda (\nabla^2 u + \chi)) = \psi
            \;\; \mbox{in $\bM$} \\
%  u = \varphi \;\; \,& \mbox{on $\partial M$}
%\end{aligned} \right.
\end{equation}
satisfying the Dirichlet boundary condition
\begin{equation}
\label{3I-12}
 u = \varphi \;\;  \mbox{on $\partial M$}
\end{equation}
%where $\partial M$ denotes the boundary of $M$.
where $f$ is a smooth symmetric function of $n$ variables,
$\chi$ a smooth $(0,2)$ tensor on $M$, $\nabla^2 u$ denotes the Hessian of 
a function $u \in C^2 (M)$, and 
$\lambda (\nabla^2 u + \chi) = (\lambda_1, \cdots, \lambda_n)$
are the eigenvalues of $\nabla^2 u + \chi$ with respect to the metric $g$. 

Equation~\eqref{3I-10} first received systematic study by 
Caffarelli-Nirenberg-Spruck~\cite{CNS3}
% who treated the Dirichlet problem in $\bfR^n$ 
whose work had been fundamentally influential to 
development of the theory and applications of fully nonlinear elliptic and parabolic equations. There were 
important contributions to the subject 
from many authors including  
Chou-Wang~\cite{CW01}, Dong~\cite{Dong06}, Ivochkina~\cite{Ivochkina85}, 
Ivochkina-Trudinger-Wang~\cite{ITW04}, 
Krylov~\cite{Krylov83a}, Li~\cite{LiYY90}, 
Trudinger~\cite{Trudinger95}, 
Trudinger-Wang~\cite{TW99}, 
Urbas~\cite{Urbas02},  Wang~\cite{Wang94}, Yuan~\cite{Yuan01}, etc. 
Li~\cite{LiYY90} first studied equation (1.1) for $\chi = g$ on closed Riemannian manifolds, 
followed by Urbas~\cite{Urbas02} and 
more recently the author~\cite{Guan14}.
Building on existing results and techniques,
our primary goal in this paper is to further understand the classical solvability of the Dirichlet problem~\eqref{3I-10}-\eqref{3I-12} in $\bfR^n$ and on general Riemannian manifolds. Specifically, we wish to find optimal conditions under which the Dirichlet problem %~\eqref{3I-10}-\eqref{3I-12} 
admits smooth solutions. 
%%%%%%%%%%%%

In order to treat the problem in the framework of elliptic theory, we follow~\cite{CNS3} and assume
the function $f$ to be defined in
a symmetric open and convex cone $\Gamma \subset \bfR^n$ with vertex
at the origin and boundary $\partial \Gamma \neq \emptyset$,
and $\Gamma_n \subseteq \Gamma$ where
\begin{equation}
\label{3I-15}
\Gamma_n \equiv
\{\lambda \in \bfR^n: \mbox{each component $\lambda_i > 0$}\}, 
%\subseteq \Gamma,
\end{equation}
and satisfy the structure conditions introduced in~\cite{CNS3} 
which have become standard in the literature: 
\begin{enumerate}
\item[{\bf (a)}]
the {\em ellipticity} condition
\begin{equation}
\label{3I-20}
f_i = f_{\lambda_i} \equiv \frac{\partial f}{\partial \lambda_i} > 0 \;\;
\mbox{in $\Gamma$}, \;\; 1 \leq i \leq n,
\end{equation}
\item[{\bf (b)}]
the {\em concavity} condition
\begin{equation}
\label{3I-30}
\mbox{$f$ is a concave function in $\Gamma$},
\end{equation}
%\end{enumerate}
%\begin{equation}
%\label{3I-40}
%\delta_{\psi, f} \equiv \inf_M \psi - \sup_{\partial \Gamma} f > 0
%   \;\; \mbox{where} \;
%  \sup_{\partial \Gamma} f \equiv \sup_{\lambda_0 \in \partial M}
%    \limsup_{\lambda \rightarrow \lambda_0} f (\lambda)
%   < \inf_{\Omega} \psi,
%\;\; \forall \, \lambda_0 \in \partial \Omega
%\end{equation}
\item[{\bf (c)}]
the {\em nondegeneracy} condition
\begin{equation}
\label{3I-40}
\delta_{\psi, f} \equiv \inf_M \psi - \sup_{\partial \Gamma} f > 0, 
\end{equation}
\end{enumerate}
where
\[ \sup_{\partial \Gamma} f \equiv \sup_{\lambda_0 \in \partial \Gamma}
                   \limsup_{\lambda \rightarrow \lambda_0} f (\lambda). \]

It was shown in \cite{CNS3} that
condition (\ref{3I-20}) implies that equation (\ref{3I-10}) is elliptic
for a solution $u \in C^2 (\Omega)$ with 
$\lambda (\nabla^2 u + \chi) \in \Gamma$; we call such functions
 {\em admissible}. Condition \eqref{3I-40} ensures that 
equation~\eqref{3I-10} is nondegenerate.
Moreover,  under assumptions \eqref{3I-20} and \eqref{3I-40}
equation~\eqref{3I-10} will become uniformly elliptic once {\em a priori}  second order derivative estimates are established for admissible solutions. 
According to \cite{CNS3}, it follows from \eqref{3I-30} that the function
$F$ defined by $F (A) = f (\lambda [A])$ is concave for
$A \in \mathcal{S}^{n\times n}$ with
$\lambda [A] \in \Gamma$, where $\mathcal{S}^{n\times n}$ is the set of
$n \times n$ symmetric matrices. 
%see \cite{CNS3}. 
Consequently,  
one can apply Evans-Krylov theorem to obtain 
$C^{2, \alpha}$ estimates
from $C^2$ bounds for admissible solutions of \eqref{3I-10}-\eqref{3I-12} 
under assumptions~\eqref{3I-20}-\eqref{3I-40}.
From this point of view 
%Therefore, 
these conditions seem indispensable to the classical solvability of the Dirichlet problem~\eqref{3I-10}-\eqref{3I-12}.
%Our main results will fail if any of these assumptions is dropped. 

In a seminal paper~\cite{CNS3}
Caffarelli-Nirenberg-Spruck
treated the Dirichlet problem~\eqref{3I-10}-\eqref{3I-12} 
in a bounded smooth domain $M$ in Euclidean space $\bfR^n$ with the
geometric property % that there exists $R > 0$ such that
\begin{equation}
\label{3I-70}
(\kappa_1, \ldots, \kappa_{n-1}, R ) \in \Gamma
%(\kappa_{\partial \Omega} (x), R) \in \Gamma, \;\; \forall \; x \in
%\partial \Omega
\;\; \mbox{on $\partial M$}
\end{equation}
for some $R > 0$, where $(\kappa_1, \ldots, \kappa_{n-1})$ are
the principal curvatures of $\partial M$. 
They proved %Theorem~\ref{3I-thm1} 
for $\chi = 0$ the classical solvability
%$f$ satisfing \eqref{3I-20}-\eqref{3I-40} and 
assuming \eqref{3I-20}-\eqref{3I-40}, \eqref{3I-70} and 
the additional conditions on $f$: 
for any $C > 0$ and compact set $K \subset \Gamma$
there is a number $R = R (C, K)$ such that
\begin{equation}
\label{3I-50}
f (\lambda_1, \ldots, \lambda_{n-1}, \lambda_n + R) \geq C
\;\;\mbox{for all $\lambda \in K$},
\end{equation}
and
\begin{equation}
\label{3I-60}
f (R \lambda) \geq C
\;\; \mbox{for all $\lambda \in K$.}
\end{equation}
Li~\cite{LiYY90} considered general $\chi$ and obtained various extensions. 
Later Trudinger~\cite{Trudinger95} was able to remove condition~\eqref{3I-50} using a new
method to derive
second order boundary estimates which will be important to our approach.

Clearly \eqref{3I-20} and \eqref{3I-60} imply that 
\begin{equation}
\label{3I-60a}
\sum f_i \lambda_i \geq 0
\;\; \mbox{in $\Gamma$.}
\end{equation}
According to \cite{CNS3}, from assumptions \eqref{3I-70} and \eqref{3I-60} one can construct {\em strict} admissible subsolutions.
For $\Gamma = \Gamma_n$ condition~\eqref{3I-70} implies that the domain $M$ is strictly convex. 
More generally, it was proved in \cite{CNS3} that when $\Gamma$ is a type 1 cone, i.e. each 
positive $\lambda_i$-axis lies on the boundary of $\Gamma$
(in \cite{Guan14} it was called type 2 by mistake),
a domain in $\bfR^n$ satisfying \eqref{3I-70} must be simply-connected.

We wish to solve the Dirichlet problem~\eqref{3I-10}-\eqref{3I-12} on general Riemannian manifolds with boundary of arbitrary geometric shape, extending the above theorem of \cite{CNS3}. We first state  
our main existence result which is optimal and new for equations in $\bfR^n$ in its generality.
In addition to \eqref{3I-20}-\eqref{3I-40}, we shall
assume % that $f$ satisfies 
\begin{equation}
\label{3I-55a'}
\sum f_i (\lambda) \lambda_i 
     \geq  - \omega_f (|\lambda^-|) \sum f_i \;\; 
\mbox{in $\Gamma (\psi)$}
\end{equation}
for some positive nondecreasing function $\omega_f$ satisfying
\[ \lim_{t  \rightarrow + \infty} 
     \frac{\omega_f (t)}{t} = 0 \]
   %\;\; \mbox{in $\Gamma (\psi)$} \]  
where $\lambda^- = (\lambda_1^-, \ldots, \lambda_n^-)$
and 
\[ \Gamma (\psi) \equiv \Gamma \cap 
    \Big\{\inf_M \psi \leq f \leq \sup_M \psi\Big\}. \]

\begin{theorem}
\label{3I-thm1}
Let $\psi \in C^{k, \alpha} (\bM)$,
$\varphi \in C^{k+2, \alpha} (\partial M)$, $k \geq 2$.
There exists a unique admissible
solution $u \in C^{k+2, \alpha} (\bM)$ of the Dirichlet problem~\eqref{3I-10}-\eqref{3I-12},
provided that \eqref{3I-20}-\eqref{3I-55a'} hold and that
there exists an admissible subsolution
$\ul{u} \in C^2 (\bM)$ satisfying
%$\lambda (\nabla^2 \ul{u}) \in \Gamma$,
\begin{equation}
\label{3I-11s}
\left\{ \begin{aligned}
f (\lambda (\nabla^2 \ul{u} + \chi)) \,& \geq \psi \;\;
             \mbox{in $\bM$} \\
\ul{u} =  \varphi \;\; \,& \mbox{on $\partial M$}.
\end{aligned} \right.
\end{equation}
If both $\psi$ and
$\varphi$ are smooth, $u \in C^{\infty} (\bM)$.
Moreover,  condition~\eqref{3I-55a'} can be removed if the sectional curvatures of $M$ are nonnegative. 
\end{theorem}

In particular, Theorem~\ref{3I-thm1} holds in 
$\bfR^n$ without assumption~\eqref{3I-55a'}.  
We expect that assumption~\eqref{3I-55a'}, which is only needed in deriving the gradient estimates, be removed completely. 
 All other conditions, however, seem necessary for 
Theorem~\ref{3I-thm1} to hold.

\begin{remark}
For fully nonlinear degenerately elliptic equations, the best 
regularity one can hope is $u \in C^{1,1} (\bM)$ in general. This is the case for the degenerate Monge-Amp\`ere equation $\det \nabla^2 u =  0$
which does not satisfy \eqref{3I-40}.
\end{remark}

\begin{remark}
The examples of 
Nadirashvili-Vladut~\cite{NV14} indicate that 
Theorem~\ref{3I-thm1} fails for nonconcave $f$ in dimension $n \geq 5$. 
\end{remark}

\begin{remark}
In Theorem~\ref{3I-thm1} the boundary 
$\partial M$ is assumed to be smooth and compact
but otherwise arbitrary; there are no geometric restrictions
on $\partial M$. Consequently, 
the Dirichlet problem may fail to have admissible 
solutions in $C^2 (\bM)$ without the subsolution assumption.
%which can be viewed as a necessary condition. 
This can easily be seen from the 
Monge-Amp\`ere equation 
$\det \nabla^2 u =  1$ in a nonconvex domain $\Omega \subset \bfR^n$ with constant boundary data. 
\end{remark}

\begin{remark}
There is an interesting example in \cite{CNS3} which 
shows that an equation in the unit ball in $\bfR^2$ satisfying 
conditions~\eqref{3I-20}-\eqref{3I-40} may have a solution
which is not Lipschitz up to the boundary. 
\end{remark}

%under more technical conditions 
%in addition to \eqref{3I-20}-\eqref{3I-40}.
% with nonnegative sectional curvature,
%and more recently by the author~\cite{Guan14}.

In order to prove Theorem ~\ref{3I-thm1} it is critical to establish the {\em a priori} $C^2$ estimate
\begin{equation}
\label{hess-a10C2}
|u|_{C^2 (\bar{M})} \leq C
\end{equation}
for admissible solutions; higher order estimates will follow from 
Evans-Krylov theorem and the classical Schauder theory 
for uniformly elliptic linear equations, and the existence can be 
proved using the standard continuity method. 
There are three key ingredients which we formulate below
in Theorem~\ref{3I-thm2}. 
In the rest of this paper we assume 
$\psi \in C^{\infty} (\bM)$ and
$\varphi \in C^{\infty} (\partial M)$. Let 
$u \in C^{4} (\bM)$ be an admissible solution of
the Dirichlet problem~\eqref{3I-10}-\eqref{3I-12}, and 
\[ \omega (u) = \sup_M u - \inf_M u. \] 

\begin{theorem}
\label{3I-thm2}
Under the assumptions \eqref{3I-20}-\eqref{3I-40} and \eqref{3I-11s}, $u$ satisfies the following estimates:
%\begin{itemize}
%\item[{\bf (i)}]
 {\bf (i)}
a maximum principle for second order derivatives
\begin{equation}
\label{hess-a10}
\max_{\bar{M}} |\nabla^2 u| \leq
 C_1 e^{C_2 \omega (u)}  \Big(1 + \max_{\bM} |\nabla u|^2 
     + \max_{\partial M}|\nabla^2 u|\Big), 
     %e^{C_2 \omega (u)} 
\end{equation}
%\end{itemize}
%\item[{\bf (ii)}]
{\bf (ii)}
the boundary estimate for second order derivatives
%Suppose moreover that $\ul{u} = \varphi$ on $\partial M$. 
\begin{equation}
\label{hess-a10b}
\max_{\partial M}|\nabla^2 u| \leq C_3 
    = C_3 (|u|_{C^1 (\bM)}),
\end{equation}
%\item[{\bf (iii)}]
{\bf (iii)}
a maximum principle for the gradient if in addition 
\eqref{3I-55a'} holds
\begin{equation}
\label{hess-a10g}
\max_{\bar{M}} |\nabla u| \leq
 C_4 %e^{C_2 \omega (u)}  
   (1 + \omega (u)) + C_5 \max_{\partial M} |\nabla u|
\end{equation}
where $C_1$, $C_2$, $C_4$ and $C_5$ are uniform constants  independent of $u$. 
\end{theorem}

\begin{remark} 
The boundary condition $u = \ul u = \varphi$ on $\partial M$ is not needed in the global estimates \eqref{hess-a10} and \eqref{hess-a10g}
\end{remark}

In comparison to work in \cite{CNS3}, there are some nontrivial new difficulties to derive these estimates on general manifolds with boundary of arbitrary geometric shape.
More technically, in deriving the estimates a critical issue is to control some key terms involving the Riemannian curvature of $M$ and in addition, for the boundary estimate 
\eqref{hess-a10b}, 
the second fundamental form of $\partial M$. 
We wish to seek methods and technical tools which enable us to overcome these difficulties and can be used to study more general equations. 
Building on previous results (see e.g. \cite{Guan14} and references therein), we 
achieve our goal by further understanding the roles that the subsolution and the concavity property~\eqref{3I-30} can play in these estimates.

In the theory of fully nonlinear elliptic equations, 
a cornerstone is Evans-Krylov theorem which crucially relies on
concavity of the equation. 
In a sense our work provides some evidence from another point of view of the importance of the concavity assumption which as we shall see plays key roles in all estimates in Theorem~\ref{3I-thm2}.
(This is somewhat different from previous approaches.)

For $\sigma \in (\hat{\sigma}, \tilde{\sigma})$ let
$\Gamma^{\sigma} = \{\lambda \in \Gamma: f (\lambda) > \sigma\}$, where
\[ \hat{\sigma} = \sup_{\partial \Gamma} f, \;\;
    \tilde{\sigma} = \sup_{\Gamma} f. \]
By conditions~\eqref{3I-20} and \eqref{3I-30} the level set  $\partial \Gamma^{\sigma} = \{\lambda \in \Gamma: f (\lambda) = \sigma\}$ 
is a smooth and convex non-compact
complete hypersurface in $\bfR^n$. 
For $\lambda \in \Gamma$ let 
\[ \nu_{\lambda} := \frac{Df (\lambda)}{|Df (\lambda)|} \]
denote the unit normal vector to 
$\partial \Gamma^{f (\lambda)}$ at $\lambda$. 
The following observation, while rather elementary, will be remarkably important to our proof of all estimates in 
Theorem~\ref{3I-thm2}.

\begin{lemma}
\label{ma-lemma-C10}
Given $\mu \in \Gamma$ and $\beta > 0$ 
there exists uniform constant $\varepsilon > 0$ such that for any $\lambda \in \Gamma$ with 
$|\nu_{\mu}- \nu_{\lambda}| \geq  \beta$, 
\begin{equation}
\label{gj-C160'}
  \sum f_i (\lambda) (\mu_i - \lambda_i)
   \geq f (\mu) - f (\lambda) + 
            \varepsilon \sum f_i (\lambda)  + \varepsilon.
\end{equation}
\end{lemma}

Since $f$ is concave this inequality %\eqref{gj-C160'}
always holds for $\varepsilon = 0$. 
In next sections we shall see how this property can be combined with the subsolution assumption to form a powerful tool for us to derive the desired estimates where both of the extra terms for $\varepsilon > 0$ will play key roles.
Note that  Lemma~\ref{ma-lemma-C10} is purely a property of the function $f$;
so it could be useful in other fields such
as convex geometry and optimal control theory. 

Subsolutions and  supersolutions are important concepts in the theory of partial differential equations.  
The classical method of subsolutions and supersolutions for semilinear elliptic equations, for instance, asserts roughly that the
existence of a subsolution and supersolution implies 
existence of solution. 
For quasilinear or fully nonlinear elliptic equations,
most of the known existence results on the Dirichlet problem also seem to
rely on existence of subsolutions. 
%which can be viewed as a
%necessary condition since a solution is also a subsolution. 
In classical papers such as \cite{Serrin69} on the equation of prescribed mean curvature, and 
~\cite{CNS1, CKNS, CNS3} on Monge-Amp\`ere and more general fully nonlinear equations, 
geometric conditions on the boundary were usually imposed
and used to construct subsolutions or local barriers near boundary. 
It is often more flexible 
in applications, especially to geometric problems, 
to use the subsolution assumption in place of geometric conditions on boundary, 
which was a major motivation for Hoffman-Rosenberg-Spruck~\cite{HRS} to introduce the idea. 
In the papers \cite{GS02} and \cite{TW02} on a Plateau type problem for hypersurafces of constance Gauss curvature,
for instance,
it was proved that a subsolution-like assumption could eliminate all topological obstructions which were shown nontrivial by Rosenberg~\cite{Rosenberg93}.

The existence of a subsolution immediately provides a lower bounds for solutions and gradient bounds on boundary. Inspired by geometric applications, 
Hoffman-Rosenberg-Spruck~\cite{HRS} first  
used subsolutions in second order boundary estimates,
and the idea was further developed by Spruck and the author~\cite{GS93, Guan98a, Guan98b} (see also \cite{GL96}) 
for real and complex Monge-Amp\`ere equations. 
The technique has found
interesting applications in geometric problems such as
Danoldson's conjecture~\cite{Donaldson99} on geodesics in the space of K\"ahler metrics (e.g.~\cite{Chen00, Blocki, PS10}) and Chern-Levine-Nirenberg conjecture~\cite{CLN69} on 
intrinsic norms \cite{GuanPF02}; see also \cite{GS02} and \cite{TW02}. 
In this paper, we make use of the subsolution combined with 
Lemma~\ref{ma-lemma-C10}
in all estimates in Theorem~\ref{3I-thm2}.  
It seems the first time for a gradient estimate to rely crucially on the concavity assumption.
Even for the Dirichlet problems in $\bfR^n$,  the boundary estimate \eqref{hess-a10b} is new in terms of its generality. 

Besides the most studied examples  
$f = \sigma_k^{\frac{1}{k}}$ and
$f = (\sigma_k/\sigma_l)^{\frac{1}{k-l}}$, 
$1 \leq l <  k \leq n$ defined in the Garding cone
\[  \Gamma_k = \{\lambda \in \bfR^n: \sigma_j (\lambda) > 0 ,\;
\mbox{$\forall$ $1 \leq j \leq k$}\} \]
where $\sigma_k$ is the $k$-th elementary symmetric
polynomial, 
there are other interesting functions which satisfy conditions~\eqref{3I-20} and \eqref{3I-30},  such as 
$f = \log P_k$ where
\[ P_k (\lambda) := \prod_{i_1 < \cdots < i_k}
(\lambda_{i_1} + \cdots + \lambda_{i_k}), \;\; 1 \leq k \leq n\]
defined in the cone
\[ \mathcal{P}_k : = \{\lambda \in \bfR^n:
      \lambda_{i_1} + \cdots + \lambda_{i_k} > 0\},  \]
and for $\alpha \leq 1$, $\alpha \neq 0$,
%\begin{equation}
%\label{3I-320}
 \[ f (\lambda) = \alpha \sum \lambda_i^{\alpha}, 
\;\; \lambda \in \Gamma_n, \]
%\end{equation}
The function 
%$f = \sigma_n^{\frac{1}{n}}$ corresponds to 
%the Monge-Amp\`ere equation, while 
$f = \log P_{n-1}$ is 
connected to a conjecture of Gauduchon in non-K\"ahler 
complex geometry and form-type Calabi-Yau equations; see e.g. \cite{FWW11, FWW15, TW17, STW17, GN}.

The current paper is a revision of \cite{Guan13a} 
which contains both Lemma~\ref{ma-lemma-C10} and  
the second order estimates \eqref{hess-a10} and 
\eqref{hess-a10b} in Theorem~\ref{3I-thm2}.
The gradient estimate \eqref{hess-a10g} 
which is substantial to the proof of Theorem~\ref{3I-thm1} 
was carried out more recently. Note that the gradient estimate does not follow from the rescaling method of Dinew-Kolodziej~\cite{DK} as it would require in place of 
\eqref{hess-a10b} 
a second order boundary estimate of the form
\begin{equation}
\label{hess-a10b'}
\max_{\partial M} |\nabla^2 u| 
    \leq C \max_{\bM} |\nabla u|^2 + C. 
\end{equation}
It would be interesting to derive such estimates. 

Since \cite{Guan13a} was posted in arXiv in early 2014 there 
have appeared several interesting papers, e.g. 
\cite{Sze18, CJY, CPW17} which seem to have been influenced by ideas from \cite{Guan14, Guan13a}.
In \cite{Sze18}, Sz\'ekelyhidi introduced a notion of 
extended subsolutions, called $\cC$-{\em subsolutions}, for equations on closed Hermitian manifolds. 
At the end of the paper we shall add an appendix to  discuss briefly his condition
for equation~\eqref{3I-10} on Riemannian manifolds, and show that it agrees for type 1 cones 
%the $\cC$-subsolution condition \cite{Sze18} is equivalent 
with a condition earlier introduced in \cite{Guan14}, 
therefore clarify relations between the two notions.

Besides the appendix the rest of this paper is divided into three sections in which 
we  establish estimates \eqref{hess-a10}, \eqref{hess-a10b} and \eqref{hess-a10g}, 
respectively.

The author wished to thank Jiaping Wang and Xiangwen Zhang for stimulating communications and especially for pointing out mistakes in our previous attempts to prove the gradient estimates, 
and to thank 
Heming Jiao, Shujun Shi and Zhenan Sui for fruitful discussions and their contributions to our joint 
papers~\cite{GJ, GSS15} where part of results in \cite{Guan13a} were extended to more general elliptic and parabolic equations.

\bigskip

\section{Global estimates for second derivatives}
\label{3I-C}
\setcounter{equation}{0}

\medskip

In this section we derive the global
second order estimates \eqref{hess-a10}.
% under hypotheses (\ref{3I-20})-(\ref{3I-40}) and \eqref{3I-11s}.
The first half closely follows the argument in
\cite{Guan14}, so we shall only give an outline for completeness while keeping track the explicit dependence on 
$|\nabla u|_{C^0 (\bM)}$,
followed by our new ideas which will also be
crucial in the following sections.
As in \cite{Guan14} one of the key points is to employ the 
subsolution in construction of test functions. 
Consider as in \cite{Guan14} 
\[ W = \max_{x \in \bar{M}} \max_{\xi \in T_x M^n, |\xi| = 1}
  (\nabla_{\xi \xi} u + \chi (\xi, \xi)) e^{\eta} \]
where $\eta$ is a function to be determined.
Suppose $W$ is achieved at an interior point
$x_0 \in M$ for some unit vector $\xi \in T_{x_0}M^n$
and choose smooth orthonormal local frames $e_1, \ldots, e_n$
about $x_0$ such that $e_1 = \xi$, $\nabla_i e_j = 0$
and $\{\nabla_{ij} u +  \chi_{ij}\}$ is diagonal at $x_0$.

Write $U = \nabla^2 u +  \chi$ and $U_{ij} = \nabla_{ij} u +  \chi_{ij}$.
At the point $x_0$ we have for $1 \leq i \leq n$, %$i = 1, \ldots, n$,
\begin{equation}
\label{hess-a30}
\frac{\nabla_i U_{11}}{U_{11}} + \nabla_i \eta = 0,
%\mbox{for each $1 \leq i \leq n$},
\end{equation}
\begin{equation}
\label{hess-a40}
\begin{aligned}
\frac{\nabla_{ii} U_{11}}{U_{11}}
   - \Big(\frac{\nabla_i U_{11}}{U_{11}}\Big)^2 + \nabla_{ii} \eta \leq 0.
\end{aligned}
\end{equation}

Next, write equation~\eqref{3I-10} in the form
\begin{equation}
\label{3I-10'}
F (U) := f (\lambda (U)) = \psi
\end{equation}
and 
\[ F^{ij} = F^{ij} (U) = \frac{\partial F}{\partial U_{ij}} (U).\]
Differentiating equation~\eqref{3I-10'} twice yields
\begin{equation}
\label{hess-a60}
F^{ij} \nabla_k U_{ij} =  \nabla_k \psi, \;\; \mbox{for all $k$},
\end{equation}
\begin{equation}
\label{hess-a65}
\begin{aligned}
F^{ij} \nabla_{11} U_{ij} + \sum F^{ij, kl} \nabla_1 U_{ij} \nabla_1 U_{kl}
 = \nabla_{11} \psi.
\end{aligned}
\end{equation}
%Throughout the paper we write $F^{ij} = F^{ij} (\{U_{ij}\})$.
Recall the formula
\begin{equation}
\label{hess-A80}
\begin{aligned}
\nabla_{ijkl} v - \nabla_{klij} v
\,& = R^m_{ljk} \nabla_{im} v + \nabla_i R^m_{ljk} \nabla_m v
      + R^m_{lik} \nabla_{jm} v \\
  + \,& R^m_{jik} \nabla_{lm} v
      + R^m_{jil} \nabla_{km} v + \nabla_k R^m_{jil} \nabla_m v.
\end{aligned}
\end{equation}
Therefore,
%By  (\ref{hess-a65}) and (\ref{hess-A80}),
\begin{equation}
\label{hess-a68}
\begin{aligned}
F^{ii} \nabla_{ii} U_{11}
 \geq \,& F^{ii} \nabla_{11} U_{ii} - C (|\nabla u| + U_{11}) \sum F^{ii}.
%  \geq \,& - F^{ij, kl} \nabla_1 U_{ij} \nabla_1 U_{kl} + \nabla_{11} \psi
% - C (1 + \lambda_1) \sum_i F^{ii}.
\end{aligned}
\end{equation}
In the proof the constant $C$, which may change from line to line, will be
independent of $|u|_{C^1 (\bM)}$.
 From (\ref{hess-a40}), (\ref{hess-a65}) and (\ref{hess-a68})
we derive
\begin{equation}
\label{hess-a71}
\begin{aligned}
U_{11} F^{ii} \nabla_{ii} \eta
  \leq \,& E - \nabla_{11} \psi +  C (|\nabla u| + U_{11}) \sum F^{ii}
\end{aligned}
\end{equation}
where
\[ E \equiv  F^{ij, kl} \nabla_1 U_{ij} \nabla_1 U_{kl}
           + \frac{1}{U_{11}} F^{ii} (\nabla_i U_{11})^2. \]

%we follow the idea of Urbas~\cite{Urbas02}.
Let $J = \{i: 3 U_{ii} \leq - U_{11}\}$. 
%$0 < s < 1$  (to be chosen) and
%\[  \begin{aligned}
%J \,& = \{i: 3 U_{ii} \leq - U_{11}\}, \;\;
%J & = \{U_{ii} > - a U_{11}, \; F^{ii} > a^{-1} F^{11}\}, \\
%K & = \{i: U_{ii} > a U_{11}, \; F^{ii} < a^{-1} F^{11}, \; i \neq 1\}, \\
%K  = \{i> 1: 3 U_{ii} > - U_{11}\}.
%  \end{aligned} \]
We may modify the estimate for $E$ in \cite{Guan14},
which uses an idea of Urbas~\cite{Urbas02},
%using the formula
%\begin{equation}
%\label{hess-A70}
% \nabla_{ijk} v - \nabla_{jik} v = R^l_{kij} \nabla_l v
%\end{equation}
to derive
\begin{equation}
\label{hess-a271}
\begin{aligned}
 E \leq \,& U_{11} \sum_{i \in J} F^{ii} (\nabla_i \eta)^2
             + C U_{11} F^{11} \sum_{i \notin J} (\nabla_i \eta)^2
                + \frac{C(1+ |\nabla u|^2)}{U_{11}} \sum F^{ii}.
\end{aligned}
\end{equation}

As in \cite{Guan14} we choose the function $\eta$ of the form
\[ \eta = \phi (t) + a (\ul{u} - u) \]
where $t = 1+ |\nabla u|^2$, $\phi$ is a positive function, $\phi' > 0$, %$\phi'' > 0$,
and $a \geq 1$ is a constant. Then
\[ \begin{aligned}
 \nabla_i \eta
 %  = \,& 2 \phi' \nabla_{k} u \nabla_{ik} u + a \nabla_i (\ul{u} - u) \\
   = \,& 2 \phi' (U_{ii} \nabla_i u - \chi_{ik} \nabla_{k} u)
        + a \nabla_i (\ul{u} - u),
 \end{aligned}   \]
\[  \begin{aligned}
  \nabla_{ii} \eta
   = \,&  2 \phi' (\nabla_{ik} u \nabla_{ik} u
          + \nabla_{k} u \nabla_{iik} u)
          + 4 \phi'' (\nabla_{k} u \nabla_{ik} u)^2 %\nabla_l u \nabla_{il} u
          + a \nabla_{ii} (\ul{u} - u).
\end{aligned} \]
Therefore, 
\begin{equation}
\label{hess-a272}
\begin{aligned}
  \sum_{i \in J} F^{ii} (\nabla_i \eta)^2
    \leq  C t (\phi')^2 \sum_{i \in J} F^{ii} U_{ii}^2 
        + C t a^2 \sum_{i \in J} F^{ii},
\end{aligned}
\end{equation}
\begin{equation}
\label{hess-a272.5}
 \sum_{i  \notin J} (\nabla_i \eta)^2
\leq  C t (\phi')^2 U_{11}^2 + C t a^2.
\end{equation}
and by \eqref{hess-a60},
 \begin{equation}
\label{hess-a273}
 \begin{aligned}
 F^{ii} \nabla_{ii} \eta
  %= \,&  2 \phi' F^{ii} (\nabla_{ik} u \nabla_{ik} u
  %       + \nabla_{k} u \nabla_{iik} u) \\
  %    &  + 2 \phi'' F^{ii} \nabla_{k} u \nabla_{ik} u \nabla_l u \nabla_{il} u
  %       + A F^{ii} \nabla_{ii} (\ul{u} - u) \\
\geq \,&  \phi'  F^{ii} U_{ii}^2 + 4 \phi'' F^{ii} (\nabla_{k} u \nabla_{ik} u)^2
          + a F^{ii} \nabla_{ii} (\ul{u} - u) \\
       &   - C \phi' |\nabla u|^2 \sum F^{ii} - C \phi' |\nabla u|.
\end{aligned}
\end{equation}

%Without loss of generality we assume 
Let $b_1 = \max_{\bM} (1 + |\nabla u|^2)$, $b = \gamma/b_1$ and
$\phi (t) = - \log (1 - b t)$, where $\gamma \in (0, 1/2]$ will be chosen small enough.
% where $b > 0$ is a constant sufficiently small so that
%$8 b (1+t)^2 \leq 1$ for all $0 \leq t \leq \max_{\bM} |\nabla u|^2$.
%\[ 2 b = \frac{1}{1 + \max_{\bM} |\nabla u|^2}. \]
We have $\phi' = \sqrt{\phi''} = b/(1-bt)$ and therefore,
%\[ \phi' = \frac{b}{1 - b t^2}, \;\;
%    \phi'' = (\phi')^2. \]%\frac{2b}{1 - b t^2}
              %+ \frac{4b^2 t^2}{ (1 - b t^2)^2}. \]
%\frac{2 b + 2 b^2 t^2}{ (1 - b t^2)^2} \]
%and therefore
%\[ \phi'' - 4 (\phi')^2 = \frac{2b}{1 - b (1+t)^2} > 0. \]
%\[ \phi'' - 4 (\phi')^2
%     = \frac{2 b - 14 b^2 t^2}{ (1 - b t^2)^2} > 0,
%   \;\; \forall \, 1 \leq t \leq b_1. \] %\max_{\bM} |\nabla u|^2.\]
%$\phi (t) = b (1+t)^2$;
%We may assume $\phi'' - 4 (\phi')^2 \geq 0$ in
%any fixed interval $[0, C_1]$ by requiring $b > 0$ sufficiently small.
combining \eqref{hess-a71} and
%\eqref{hess-a271}, \eqref{hess-a272},\eqref{hess-a272.5} and
 \eqref{hess-a271}-\eqref{hess-a273},
%and \eqref{hess-a274}
 \begin{equation}
\label{hess-a279}
 \begin{aligned}
\phi'  F^{ii} U_{ii}^2 +  a F^{ii} \nabla_{ii} (\ul{u} - u)
    \leq \,&  - \frac{\nabla_{11} \psi}{U_{11}}  
                  + C t a^2 \sum_{i \in J} F^{ii}  +  C t a^2 F^{11} \\
                + C \phi' |\nabla u|  + C (1 + t \phi' + t U_{11}^{-2}) \sum F^{ii}   
       \,&  + C t (\phi')^2 F^{ii} U_{ii}^2 
\end{aligned}
\end{equation}
where %$t = 1 + |\nabla u|^2$ and
$C$ is independent of $|\nabla u|_{C^0 (\bM)}$.

So far we have essentially followed \cite{Guan14} except
the choice of  function $\phi$ in order to obtain the desired 
dependence on $|\nabla u|_{C^0 (\bM)}$. Our new ideas
in the proof are present below.

Write $\mu (x) = \lambda (\nabla^2 \ul u (x) +  \chi (x))$ and
note that $\{\mu (x): x \in \bM\}$ is a compact subset of $\Gamma$.
There exists a uniform constant 
$\beta \in (0, \frac{1}{2 \sqrt{n}})$ such that
\begin{equation}
\label{gj-C145}
%\mu^{\epsilon} (x) \in \Gamma, \;\;
     \nu_{\mu (x)} -  2 \beta {\bf 1} \in \Gamma_n,
     \;\;  \forall \, x \in \bM.
\end{equation}
Recall that for $\lambda \in \Gamma$,
$\nu_{\lambda} = Df (\lambda)/|Df (\lambda)|$ 
is the unit normal vector to $\partial \Gamma^{f (\lambda)}$.
%the level hypersurface 
%\[ \partial \Gamma^{f (\lambda)} 
 %     = \{\zeta \in \Gamma: f (\zeta) = f (\lambda)\} \]
 
Let $\mu = \mu (x_0)$ and $\lambda = \lambda (U (x_0))$.
We first consider the case 
$|\nu_{{\mu}}- \nu_{{\lambda}}| \geq  \beta$.
We apply Lemma~\ref{ma-lemma-C10} 
to obtain from \eqref{hess-a279}
either
 \begin{equation}
\label{hess-a279''}
 \begin{aligned}
\varepsilon a - C \phi' |\nabla u| + 
(\varepsilon a - C (1 + t \phi' + t U_{11}^{-2})) \sum F^{ii} 
    \leq 0 
\end{aligned}
\end{equation}
or 
\begin{equation}
\label{hess-a279'}
  \phi' (1 - C t \phi')  F^{ii} U_{ii}^2 
         \leq  C t a^2 \sum_{i \in J} F^{ii}
           + C t a^2 F^{11}.
\end{equation}
Here we have used the fact (\cite{Guan14})
\[ F^{ii} (\nabla_{ii} \ul u - \nabla_{ii} u) \geq
   \sum f_i (\mu_i - \lambda_i). \]
Note that
\begin{equation}
\label{hess-a274}
 F^{ii} U_{ii}^2 \geq  F^{11} U_{11}^2 + \sum_{i \in J} F^{ii} U_{ii}^2
   \geq F^{11} U_{11}^2 + \frac{U_{11}^2}{9} \sum_{i \in J} F^{ii}.
 \end{equation}
Fixing $a$ sufficiently large and $\gamma$ sufficiently small, 
both independent of $|\nabla u|_{C^0 (\bM)}$,
we obtain either $U_{11} (x_0) \leq C \sqrt{b_1}$ from \eqref{hess-a279''} 
or $U_{11} (x_0) \leq C b_1$ from \eqref{hess-a279'}.  

Suppose now that 
$|\nu_{\mu}- \nu_{\lambda}| <  \beta$.
It follows that
$\nu_{\lambda} - \beta {\bf 1} \in \Gamma_n$ and
therefore
\begin{equation}
\label{hess-a276}
 F^{ii} \geq \frac{\beta}{\sqrt{n}} \sum F^{kk},
    \;\; \forall \, 1 \leq i \leq n.
\end{equation}
Since $\sum F^{ii} (\nabla_{ii} \ul u - \nabla_{ii} u) \geq 0$
by the concavity of $f$,
we obtain from \eqref{hess-a279} and \eqref{hess-a276}  (when $\gamma$
is small enough),
 \begin{equation}
\label{hess-a278}
 \begin{aligned}
\frac{\beta}{\sqrt{n}}  |{\lambda}|^2 \sum F^{ii}
\leq \sum F^{ii} {\lambda}_i^2 %U_{ii}^2
    \leq \,& C b_1^2 a^2 \sum F^{ii}  + C b_1
\end{aligned}
\end{equation}
where $|{\lambda}|^2  = \sum U_{ii}^2$.
 By the concavity of $f$ again,
\[ \begin{aligned}
  |{\lambda}| \sum F^{ii}
  \geq \,& f (|{\lambda}| {\bf 1}) - f ({\lambda})
        + \sum F^{ii} {\lambda}_i \\
  \geq \,& f (|{\lambda}| {\bf 1}) - f ({\mu})
                 - \frac{1}{4 |{\lambda}|}
             \sum F^{ii} {\lambda}_{i}^2 - |{\lambda}| \sum F^{ii}.
    \end{aligned} \]
Therefore,
\begin{equation}
\label{hess-a300}
 |{\lambda}|^2 \sum F^{ii} \geq \frac{|{\lambda}|}{2}
      (f (|{\lambda}| {\bf 1}) - f ({\mu}))
         - \frac{1}{8} \sum F^{ii} {\lambda}_{i}^2.
\end{equation}

Suppose $|{\lambda}| \geq 1 + \max_{x \in \bM} |\mu (x)|
\equiv \Lambda$
and let
\[ b_0 \equiv
    f (\Lambda {\bf 1}) - \max_{x \in \bM} f (\mu (x)) > 0. \]
We derive from \eqref{hess-a278} and \eqref{hess-a300} that
\begin{equation}
\label{hess-a310}
 |{\lambda}|^2 \sum F^{ii} + b_0 |{\lambda}|
  \leq C b_1^2 a^2 \sum F^{ii}  + C b_1.
\end{equation}
This gives a bound $|{\lambda}|\leq C b_1$.
The proof of \eqref{hess-a10} is complete.

\bigskip

\section{Second order boundary estimates}
\label{3I-B}
\setcounter{equation}{0}

\medskip

In this section we establish the boundary estimate \eqref{hess-a10b}.
We shall continue to use notations from the
previous section, and assume throughout the section
 that the function $\varphi \in C^{\infty} (\partial M)$ is extended
smoothly to $\bM$, still denoted $\varphi$.

For a point $x_0$ on $\partial M$, we shall choose
smooth orthonormal local frames $e_1, \ldots, e_n$ around $x_0$ such that $e_n$ is the interior normal to $\partial M$ along the boundary.
Let $\rho (x)$ and  $d (x)$ denote the distances from $x \in \bM$ to $x_0$
and $\partial M$, respectively.
%\[ \rho (x) \equiv \mbox{dist}_{M^n} (x, x_0),  \;\; d (x) \equiv %\mbox{dist} (x, \partial M)\]
We may choose $\delta_0 > 0$ sufficiently small such that
$\rho$ and $d$
%the distance function to $\partial M$
are smooth in $M_{\delta_0} = \{x \in M : \rho (x) < \delta_0 \}$.
By a straightforward calculation (see also \cite{Guan14}),
%for all $1 \leq k \leq n$,
\begin{equation}
\label{hess-E170}
\begin{aligned}
|F^{ij} \nabla_{ij} \nabla_k u|
%\leq \,& 2 F^{ij}  \Gamma_{ik}^l \nabla_{jl} u  + C \Big(1 + \sum F^{ii}\Big) \\
\leq \,& C \Big(1 + \sum f_i |\lambda_i|\Big) + C (1 + |\nabla u|) \sum f_i,
%\leq \,& \epsilon \sum f_i \lambda_i^2 + \frac{C}{\epsilon} \sum f_i + C,
 \;\; \forall \; 1 \leq k \leq n.
\end{aligned}
\end{equation}
%where $\lambda = \lambda (U)$.

From the boundary condition $u = \varphi$
on $\partial M$ we derive directly
the pure tangential second derivative bound
\begin{equation}
\label{hess-E130}
|\nabla_{\alpha \beta} u (x_0)| \leq  C (1 + |\nabla u|),  \;\; \forall \; 1 \leq \alpha, \beta < n.
%\;\;\mbox{on} \; \partial M
\end{equation}
To estimate the rest of
%mixed tangential-normal and pure normal
second derivatives we use  the following barrier function
%as in \cite{Guan12a},
\begin{equation}
\label{hess-E176}
 \varPsi
   = A_1 (u - \ul{u})  + A_2 \rho^2 + \sqrt{b_2} - w + A_3 (d - N d^2)
    %\sum_{l< n} |\nabla_l (u - \varphi)|^2
             % \pm \nabla_{\alpha} (u - \varphi)
\end{equation}
where %$b_2 = \max_{\bM} (1 + |\nabla (u - \varphi)|^2)$ and 
\begin{equation}
\label{ma-E85}
\left\{
%\begin{cases}
\begin{aligned}
%v = \,&  d - \frac{N d^2}{2}, \\
b_2 = & \max_{\bM} (1 + |\nabla (u - \varphi)|^2) \\
w = \,& \Big(b_2 + \sum_{l< n} |\nabla_l (u - \varphi)|^2\Big)^{\frac{1}{2}}.
\end{aligned}
%\end{cases}
\right.
\end{equation}
%As in \cite{Guan12a}

The following lemma is key to our proof.

\begin{lemma}
\label{ma-lemma-E10}
%Assume that \eqref{3I-20}-\eqref{3I-40}
%and \eqref{3I-11s} hold.
For any constant $K > 0$ and a function $h \in C (\ol{M_{\delta_0}})$ satisfying $h \leq C \rho^2$
on $\ol{M_{\delta_0}} \cap \partial M$ and $h \leq C$
on $\ol{M_{\delta_0}}$, 
there exist uniform positive constants
$t, \delta$ sufficiently small, and $A_1$, $A_2$, $A_3$,
$N$ sufficiently large
such that $\varPsi \geq h$ on $\partial M_{\delta}$ and
\begin{equation}
\label{ma-E86}
F^{ij} \nabla_{ij} \varPsi  \leq - K  \Big(1 + \sum f_i |\lambda_i|\Big) 
  + K (1 + |\nabla u|) \sum f_i \;\; \mbox{in $M_{\delta}$}.
\end{equation}
\end{lemma}

\begin{proof}
For convenience we write $\tu = u - \varphi$.
At a point $x \in M_{\delta}$ we derive, by \eqref{hess-E170} 
% we assume that $U_{ij}$ and $F^{ij}$ are both diagonal.
\begin{equation}
\label{hess-E170'}
\begin{aligned}
F^{ij} \nabla_{ij} w = 
     \,& \sum_{l < n}  \frac{\nabla_l \tu}{w} F^{ij} \nabla_{ij} \nabla_l \tu
          + \frac{1}{w} \sum_{k, l < n}  F^{ij} \Big(\delta_{kl} - \frac{\nabla_k \tu \nabla_l \tu}{w^2}\Big)  \nabla_{il} \tu \nabla_{jk} \tu \\
%  = \,& 2  F^{ij} \nabla_{\beta} (u - \varphi)
%          \nabla_{ij} \nabla_{\beta} (u - \varphi) \\
 %    \,& + 2  F^{ij} \nabla_i \nabla_l (u - \varphi)
 %              \nabla_j \nabla_{\beta} (u - \varphi)  \\
\geq \,& \frac{1}{4 w} \sum_{l < n} F^{ij} U_{i l} U_{j l}           
     - C \Big(1 + \sum f_i |\lambda_i|\Big)  
         - C (1 + |\nabla u|) \sum f_i.
%\geq \,& F^{ij} U_{i \beta} U_{j \beta} - \epsilon \sum f_i \lambda_i^2
%         - C \Big(1 + \sum f_i\Big).
\end{aligned}
\end{equation}
By Proposition~2.19 in \cite{Guan14} there exists an index $r$ such
that
\begin{equation}
\label{hess-E175}
\sum_{l < n} F^{ij} U_{i l} U_{j l}
     \geq \frac{1}{2} \sum_{i \neq r} f_i \lambda_i^2.
\end{equation}

At a fixed point in $M_{\delta}$ we consider two cases:
{\bf (a)}
$|\nu_{\mu}- \nu_{\lambda}| <  \beta$
 and
{\bf (b)}
$|\nu_{\mu}- \nu_{\lambda}| \geq \beta$
where ${\mu} = \lambda (\nabla^2 \ul u +  \chi)$,
$\lambda = \lambda (\nabla^2 u + \chi)$ and
$\beta$ is as in \eqref{gj-C145}.

Case {\bf (a)}
$|\nu_{\mu} - \nu_{\lambda}| <  \beta$.
We first show that
\begin{equation}
\label{hess-E177}
\sum_{i \neq r} f_i \lambda_i^2 %\sum_{l < n} F^{ij} U_{i l} U_{j l}
     \geq c_0 \sum f_i \lambda_i^2 - C_0 \sum f_i
\end{equation}
for some $c_0, C_0 > 0$. 
For this we shall use
\begin{equation}
\label{hess-a276'}
 f_i \geq \frac{\beta}{\sqrt{n}} \sum f_k,
    \;\; \forall \, 1 \leq i \leq n
\end{equation}
by \eqref{hess-a276} and the fact
 $\sum \lambda_i \geq 0$ which implies
\begin{equation}
\label{hess-a27}
 \sum_{\lambda_i < 0} \lambda_i^2
     \leq \Big(- \sum_{\lambda_i < 0} \lambda_i\Big)^2
    \leq n \sum_{\lambda_i > 0} \lambda_i^2.
\end{equation}
By \eqref{hess-a276'} and \eqref{hess-a27},
\[ f_r \lambda_r^2 \leq n f_r \sum_{\lambda_i > 0} \lambda_i^2
    \leq \frac{n \sqrt{n}}{\beta} \sum_{\lambda_i > 0} f_i \lambda_i^2 \]
provided that $\lambda_r < 0$.
Suppose now that $\lambda_r > 0$.
By the concavity of $f$,
\[ f_r \lambda_r \leq f_r \mu_r 
        + \sum_{i \neq r} f_i (\mu_i - \lambda_i). \]
It follows from Schwarz inequality that
\[  \begin{aligned}
  \frac{\beta f_r \lambda_r^2}{\sqrt{n}} \sum f_k
  \leq \,& f_r^2 \lambda_r^2
  \leq    2 f_r^2 \mu_r^2
        + 2 \sum_{k \neq r} f_k  \sum_{i \neq r} f_i (\mu_i - \lambda_i)^2 \\
  \leq \,& 4 \Big(\sum_{i \neq r} f_i \lambda_i^2
              + \sum f_i \mu_i^2\Big) \sum f_k.
\end{aligned}    \]
This proves \eqref{hess-E177}.

By \eqref{hess-E177} and Schwarz inequality we have
\begin{equation}
\label{hess-E194}
 \begin{aligned}
 \sum_{i \neq r} f_i \lambda_i^2
\geq \,& 2 A w \sum f_i |\lambda_i|
               - \Big(\frac{A^2 w^2}{c_0} + C_0\Big) \sum f_i \\
\geq \,& 2 A w \sum f_i |\lambda_i| - C_1 A^2 w^2 \sum f_i.
\end{aligned}  
\end{equation}
Consequently, in view of \eqref{hess-E170'}, \eqref{hess-E175} and 
\eqref{hess-E177} we may fix $A = O(K)$ to obtain
%we choose $A_3 \geq 1$ such that
\begin{equation}
\label{hess-E195}
\begin{aligned}
F^{ij} \nabla_{ij} w
\geq \,& \frac{c_0}{16 w} \sum f_i \lambda_i^2  + 
\frac{A - C}{8} \sum f_i |\lambda_i|
%+ \frac{b_0 |\lambda|}{8 w} 
- C A^2 w \sum f_i - C \\
\geq \,& \frac{c_0}{16 w} \sum f_i \lambda_i^2  + 
K \sum f_i |\lambda_i|
%+ \frac{b_0 |\lambda|}{8 w} 
- C A^2 w \sum f_i - C. 
\end{aligned}
\end{equation}
%provided that $A = O (1)$ and $R_0 = O (w)$ are fixed sufficiently large. 

Next, since $|\nabla d| \equiv 1$ by \eqref{hess-a276'} 
we may fix $N = O(1)$ sufficiently large and then 
$\delta \leq 1/N$ such that in $M_{\delta}$,
\begin{equation}
\label{eq-110}
\begin{aligned}
F^{ij} \nabla_{ij} (d - N d^2) 
\leq \,& C (1+2N d) \sum F^{ii}
        - 2 N  F^{ij} \nabla_{i} d \nabla_{j} d \\
   \leq \,& - \frac{2\beta N - C}{\sqrt{n}} \sum f_i 
   \leq  - \sum f_i.
% \;\; \mbox{in $M_{\delta}$.}
\end{aligned}
\end{equation}
% \;\; \mbox{in $M_{\delta}$  for $\delta \leq 1/N$.}
%and that $F^{ij}  \nabla_{ij} (u - \ul{u}) \leq 0$ by the concavity of $f$.
  %we see that
Now we  choose $A_2 = O(\delta^{-2} \sqrt{b_2})$
such that 
\[ A_2 \rho^2 + \sqrt{b_2} - w \geq h 
%A_3 \sum_{l< n} |\nabla_l (u - \varphi)|^2
    \;\; \mbox{on $\partial M_{\delta}$}. \]
As $F^{ij}  \nabla_{ij} (u - \ul u) \leq 0$ by the concavity of $f$, it follows from
\eqref{hess-E195} and \eqref{eq-110}, 
\begin{equation}
\label{hess-E198}
\begin{aligned}
F^{ij} \nabla_{ij} \varPsi 
\leq \,& - \frac{c_0}{16 w} \sum f_i \lambda_i^2  - K \sum f_i |\lambda_i| \\
     \,& + (C A^2 w + C A_2 - A_3) \sum f_i + C. 
\end{aligned}
\end{equation}

By \eqref{hess-a300} and \eqref{hess-a276'} there are constants 
$b_0, \, R_0 > 0$ such that when $|\lambda| \geq R_0$
\begin{equation}
\label{hess-E195'}
|\lambda|^2 \sum f_i \geq \sum f_i \lambda_i^2 \geq b_0 |\lambda|. 
\end{equation}
 % for $R$ sufficiently large.
On the other hand, if $|\lambda| \leq R_0$ then 
there exists a uniform $c_1 > 0$ (depending on $R_0$) such that 
$\{F^{ij}\} \geq c_1 I$ and, in particular,  
\begin{equation}
\label{hess-E193}
\sum f_i \geq n c_1.
\end{equation}
Therefore, 
by \eqref{hess-E198}-\eqref{hess-E193} we see that \eqref{ma-E86} holds
in case {\bf (a)} if we fix $A_3 = O(K \sqrt{b_2})$ sufficiently large. Indeed, 
this is clear %from \eqref{hess-E198} and \eqref{hess-E193} if  
if either $|\lambda| \leq R_0$ or $|\lambda| \geq 16 w (K + C)/b_0 c_0$,
while when $R_0 < |\lambda| < 16 w (K + C)/b_0 c_0$ 
we have by \eqref{hess-E195'}
\[ \frac{A_3}{2} \sum f_i \geq \frac{b_0 A_3}{2 |\lambda|} \geq 
\frac{b_0^2 c_0 A_3}{32 w (K + C)} \geq K + C \]
provided that $A_3 \geq 64 \sqrt{b_2} (K + C)^2/b_0^2 c_0$.

Case {\bf (b)}
$|\nu_{\mu} - \nu_{\lambda}| \geq  \beta$.
By Lemma~\ref{ma-lemma-C10}
\begin{equation}
\label{ma-E87}
 F^{ij}  \nabla_{ij} (\ul u - u) \geq \sum f_i (\mu_i - \lambda_i)
  \geq \varepsilon \Big(1 + \sum f_i\Big) 
\end{equation}
for some $\varepsilon > 0$. 
Note that 
if $\lambda_r > 0$ in \eqref{hess-E175} then  
\begin{equation}
\label{ma-E89}
  \begin{aligned}
   \sum f_i |\lambda_i| 
    = \,& \sum  f_i \lambda_i - 2\sum_{\lambda_i < 0} f_i \lambda_i \\
        \leq \,& F^{ij} \nabla_{ij} (u - \ul u) + C \sum f_i  
                - 2 \sum_{\lambda_i < 0} f_i \lambda_i \\
\leq \,& %F^{ij} \nabla_{ij} (u - \ul u)  + 
         \frac{\epsilon}{w} \sum_{i \neq r} f_i \lambda_i^2 
            + \frac{C w}{\epsilon} \sum f_i 
     \end{aligned} 
\end{equation}
for any $\epsilon > 0$. 
Similarly, if $\lambda_r < 0$ then  
\begin{equation}
\label{ma-E89'}
\begin{aligned}
\sum f_i |\lambda_i| 
    = \,& 2 \sum_{\lambda_i > 0} f_i \lambda_i - \sum  f_i \lambda_i \\
%\leq \,& 2 \sum_{\lambda_i > 0} f_i \lambda_i
%             - F^{ij} v_{ij} + C \sum f_i \\
\leq \,& \frac{\epsilon}{w} \sum_{i \neq r} f_i \lambda_i^2
  + \frac{C w}{\epsilon} \sum f_i - F^{ij} \nabla_{ij} (u - \ul u).
     \end{aligned} 
\end{equation}
Therefore, by \eqref{hess-E170'}, \eqref{hess-E175}, \eqref{ma-E89}
and \eqref{ma-E89'}
\begin{equation}
\label{hess-E198'}
\begin{aligned}
F^{ij} \nabla_{ij} \varPsi 
\leq \,& A_1 F^{ij} \nabla_{ij} (u - \ul u) + (C w + C A_2 + C A_3 ) \sum f_i  \\
          & - \frac{1}{2w} \sum_{i \neq r} f_i \lambda_i^2 
             + C \sum f_i |\lambda_i| + C \\
\leq \,& (A_1 - C) F^{ij} \nabla_{ij} (u - \ul u) - K \sum f_i |\lambda_i| \\
          & + C (1 + \epsilon^{-2} + K) w \sum f_i + C
%- \,& \frac{1}{2w} \sum_{i \neq r} f_i \lambda_i^2 
\end{aligned}
\end{equation}
 when $\epsilon$ is sufficiently small. 
Finally, for $A_1 = O (K \sqrt{b_2})$ sufficiently large, 
we derive \eqref{ma-E86} from
\eqref{hess-E198'} and \eqref{ma-E87} in case~{\bf (b)}.
The proof of Lemma~\ref{ma-lemma-E10} is complete. 
\end{proof}

\begin{remark}
\label{remark-B10}
In a previous draft \cite{Guan13a} of this paper, \eqref{hess-a10b}
was proved under the following additional condition:
\begin{equation}
\label{3I-55'}
\sum f_i (\lambda) \lambda_i \geq  - K_0 \Big(1 + \sum f_i\Big)
\;\; \mbox{in $\Gamma \cap \{\inf_M \psi \leq f \leq \sup_M \psi\}$}
\end{equation}
for some $K_0 \geq 0$, which was used 
to derive in place of  \eqref{ma-E89'} when $\lambda_r < 0$ 
\begin{equation}
\label{ma-E89'old}
\begin{aligned}
\sum f_i |\lambda_i| 
    = \,& 2 \sum_{\lambda_i > 0} f_i \lambda_i - \sum  f_i \lambda_i 
%\leq \,& 2 \sum_{\lambda_i > 0} f_i \lambda_i
%             - F^{ij} v_{ij} + C \sum f_i \\
\leq \epsilon \sum_{i \neq r} f_i \lambda_i^2
  + \frac{C}{\epsilon} \sum f_i + C.
     \end{aligned} 
\end{equation}
It was removed using an idea of Heming Jiao to whom we wish to 
express our gratitude. 
 Jiao's idea has been incorporated in our joint paper \cite{GJ}.
\end{remark}

Applying Lemma~\ref{ma-lemma-E10} to $h = \pm \nabla_{\alpha}(u -\phi)$, 
by  \eqref{hess-E170} we derive a bound for the mixed tangential-normal
derivatives
\begin{equation}                                
\label{hess-E130'}
|\nabla_{n\alpha} u (x_0)| %\leq \nabla_n \varPsi (x_0) + C 
 \leq  A_1 \nabla_n (u - \ul u)  (x_0) + C \leq C (1 + |\nabla u|^2),
\;\; \forall \; \alpha < n.
\end{equation}

It remains to establish the double normal derivative estimate
\begin{equation}
\label{cma-200}
 |\nabla_{n n} u (x_0)| \leq C.
\end{equation}

For a $(0, 2)$ tensor $W$ on $\bM$ and $x \in \partial M$,
let $\tilde{W} (x)$ denote the restriction of $W$ to $T_x \partial M$
and $\lambda' (\tilde{W})$ the
eigenvalue of $\tilde{W}$ (with respect to the induced metric).
We shall prove \eqref{cma-200} by showing that there are uniform
%positive
constants $c_0, R_0 > 0$
such that
$(\lambda' (\tilde{U} (x)), R) \in \Gamma$ and
\begin{equation}
\label{cma-205}
f (\lambda' (\tilde{U} (x)), R) \geq \psi (x) + c_0
\end{equation}
for all $R > R_0$ and $x \in \partial M$.
Assuming such $c_0$ and $R_0$ have been found,
by Lemma 1.2 of \cite{CNS3}
we can find $R_1 \geq R_0$ from  
\eqref{hess-E130} and
\eqref{hess-E130'} such
that $U_{nn} (x_0) > R_1$ would imply
%the eigenvalues of
%$\lambda [U\}] = (\lambda_1, \cdots, \lambd'_n)$ are given by
%\[ \lambda_i = \lambda'_i + o (1/U_{nn}), \;\; 1 \leq i \leq n-1; \;\;
% \lambda_n = U_{nn} + o (1/U_{nn})
\[ f (\lambda (U))
    \geq f (\lambda' [\{\tilde{U}], U_{nn}) - \frac{c_0}{2}
    \geq \psi (x) + \frac{c_0}{2}. \]
By equation~\eqref{3I-10} this gives 
We otain a bound $U_{nn} (x_0) \leq R_1$ consequently. 
%for otherwise, we would have a controdiction:
%\[ f (\lambda (U (x_0))) \geq \psi (x_0) + \frac{c_0}{2}. \]

In order to prove \eqref{cma-205} we adapt some idea of Trudinger~\cite{Trudinger95} as in \cite{Guan14}. 
Let 
\[ m_R \equiv \min_{x \in \partial M} %\frac{1}{\psi (x_0)}
         \{f (\lambda' (\tilde{U} (x)), r) - \psi (x)\} \]
and consider 
\[ \tilde{m} \equiv \lim_{R \rightarrow \infty}  m_R. \]

  %   [f (\lambda' (\tilde{U} (x)), R)  - \psi (x)]. \]
We first show that $\tilde{m} \geq c_0$ for some uniform 
constant $c_0 > 0$;
we assume $\tilde{m} < \infty$  for otherwise we are done.

Suppose $\tilde{m}$ is achieved at a point $x_0 \in \partial M$
and choose local orthonormal frames $(e_1, \ldots, e_n)$ around $x_0$
as before, and in particular $e_n$ is normal to $\partial M$, 
such that
$U_{\alpha \beta} (x_0)$ ($1 \leq \alpha, \beta \leq n-1$) is diagonal.
For a symmetric $(n-1)^2$ matrix $\{r_{\alpha  {\beta}}\}$ with
$(\lambda' (\{r_{\alpha \beta} (x_0)\}), R) \in \Gamma$ 
for $R$ sufficiently large, we define
\begin{equation}
\label{c-20a}
  \tF_R [r_{\alpha \beta}] 
    =  f (\lambda' (\{r_{\alpha \beta}\}), R) 
\end{equation}
where $\lambda' (\{r_{\alpha \beta}\})$
denotes the eigenvalues of 
$\{r_{\alpha \beta}\}$ ($1 \leq \alpha, \beta \leq n-1$),
and
\begin{equation}
\label{c-30a}
  \tF [r_{\alpha \beta}]  = \lim_{R \rightarrow + \infty} 
    \tF_R [r_{\alpha \beta}].
\end{equation}
Note that  $\tF$ is finite and concave since $f$ is concave and continuous.

By the concavity of $\tF$ there is a symmetric matrix
$\{\tF^{\alpha {\beta}}_0\}$ such that
\begin{equation}
\label{c-200}
\tF^{\alpha {\beta}}_0 
           (r_{\alpha  {\beta}} - U_{\alpha {\beta}} (x_0)) \geq 
  \tF [r_{\alpha  {\beta}}] - \tF [U_{\alpha  {\beta}} (x_0)]
 \end{equation}
for any symmetric matrix $\{r_{\alpha  \beta}\}$ with
$(\lambda' [\{r_{\alpha  \beta}\}], R) \in \Gamma$;
%$\{\tF^{\alpha {\beta}}_0\}$ is uniquely defined if and only 
if $\tF$ is differentiable at $U_{\alpha {\beta}} (x_0)$ then 
\[ \tF^{\alpha {\beta}}_0
 = \frac{\partial \tF}{\partial r_{\alpha {\beta}}}
        [U_{\alpha {\beta}} (x_0)]. \]

Let
$\sigma_{\alpha {\beta}} = \langle \nabla_{\alpha} e_{\beta}, e_n \rangle$.
 Since $u - \ul u = 0$ on $\partial M$,
\begin{equation}
\label{c-220}
 U_{\alpha {\beta}} - \ul{U}_{\alpha {\beta}}
    = - \nabla_n (u - \ul{u}) \sigma_{\alpha {\beta}}
\;\; \mbox{on $\partial M$}.
\end{equation}
It follows that
\[ \begin{aligned}
 \nabla_n (u - \ul{u}) \tF^{\alpha {\beta}}_0 \sigma_{\alpha {\beta}} (x_0)
 %  = \,& \tF^{\alpha {\beta}}_0 ( \ul{U}_{\alpha {\beta}} (x_0) -
  %       U_{\alpha  {\beta}} (x_0)) \\
\geq \,& \tF[\ul{U}_{\alpha {\beta}} (x_0)] - \tF[U_{\alpha {\beta}} (x_0)]
%  =  \,& \tF[\ul{U}_{\alpha {\beta}} (x_0)] - \psi(x_0) - m_R
 \geq \tilde{c} - \tilde{m}
\end{aligned} \]
where
%\[  c_R \equiv \inf_{\partial M}
%  (\tF[\ul{U}_{\alpha {\beta}}] - F [\ul{U}_{ij}]) > 0 \]
\[ \tilde{c} = 
%\frac{1}{\psi (x_0)}
 %\tF [U_{\alpha \beta} (x_0)] - \psi (x_0). \]
     \liminf_{R \rightarrow \infty} c_R > 0 \]
and
\[  c_R \equiv \min_{\partial M}
    \{f (\lambda' (\tilde{\ul U}), R)  
          - f (\lambda (\tilde{\ul U}))\}. \]
Consequently, if
\[ \nabla_n (u - \ul{u}) (x_0)
    \tF^{\alpha {\beta}}_0 \sigma_{\alpha {\beta}} (x_0)
   \leq \frac{\tc}{2} \]
then $\tm\geq \tc/2 > 0$ and we are done.

Suppose now that
\[ \nabla_n (u - \ul{u}) (x_0) \tF^{\alpha {\beta}}_0 \sigma_{\alpha {\beta}} (x_0)
    > \frac{\tc}{2}. \]
%\tF^{\alpha {\beta}}_0 \ul{U}_{\alpha {\beta}} (x_0). \]
Let $\eta \equiv \tF^{\alpha {\beta}}_0 \sigma_{\alpha {\beta}}$ and note that for some uniform $\epsilon_1 > 0$, %independent of $R$.
\begin{equation}
\label{c-230}
\eta (x_0) \geq \frac{\tc}{2 \nabla_n (u - \ul{u}) (x_0)}
                \geq 2 \epsilon_1 \tc.
\end{equation}
We may assume $\eta \geq \epsilon_1 \tc$ on $\bar{M_{\delta}}$
by requiring $\delta$ small. %(depending on $R$ which is fixed).
Define in $M_{\delta}$,
\[ \begin{aligned}
\varPhi
     = \,& - \nabla_n (u - \varphi)
        + \frac{1}{\eta} \tF^{\alpha {\beta}}_0
           (\nabla_{\alpha {\beta}} \varphi + \chi_{\alpha {\beta}}
          - U_{\alpha {\beta}} (x_0)) - \frac{\psi - \psi (x_0)}{\eta}  \\
\equiv \,& - \nabla_n (u - \varphi) + Q.
\end{aligned} \]
%$Q$ are both smooth, and
We have $\varPhi (x_0) = 0$ and
$\varPhi \geq 0$ on $\partial M$ near $x_0$. 
By \eqref{hess-E170},
%By \eqref{gblq-B50} and \eqref{gblq-B190},
\begin{equation}
\label{gblq-B360}
 \begin{aligned}
F^{ij} \nabla_{ij} \varPhi
  \leq \, & - F^{ij} \nabla_{ij} \nabla_n u + C \sum F^{ii} \\
  \leq  \,& C + C \sum f_i (|\lambda_i| +1).
\end{aligned}
\end{equation}

By \eqref{ma-E86} in Lemma~\ref{ma-lemma-E10},
 for $A_1 \gg A_2 \gg A_3 \gg 1$ we derive
$\varPsi + \varPhi \geq 0$ on $\partial M_{\delta}$ and
\begin{equation}
\label{cma-106}
F^{ij} \nabla_{ij} (\varPsi + \varPhi) \leq  0 \;\; \mbox{in $M_{\delta}$}.
\end{equation}
By the maximum principle,
$\varPsi +  \varPhi \geq 0$ in $M_{\delta}$ and therefore
$\varPhi_n (x_0) \geq - \nabla_n \varPsi (x_0) \geq -C$.
This gives
%\begin{equation}
%\label{cma-310}
$ \nabla_{nn} u (x_0) \leq C$. %\frac{C}{\eta (x_0)} \leq \frac{C}{c_R}.
%\end{equation}

So we have an {\em a priori}
upper bound for all eigenvalues of $U (x_0)$.
Consequently, $\lambda (U (x_0))$ is contained in a
compact subset of $\Gamma$ by \eqref{3I-40}. 
By Lemma 1.2 of \cite{CNS3} again we obtain
\[ \begin{aligned} 
 \tm = \,& \lim_{R \rightarrow \infty} 
  f (\lambda' (\tilde{U} (x_0)), U_{nn} (x_0) + R) - \psi (x_0) \\
    = \,& \lim_{R \rightarrow \infty} 
         f (\lambda (U (x_0)) + R e_n) - \psi (x_0) \\
\geq \,& 
  f (\lambda (U (x_0)) + R_0 \ve_n) - \psi (x_0)] > 0
 %\equiv c_0 > 0 %\frac{\tilde{m}_R}{2} 
\end{aligned} \]
for $R_0$ sufficiently large, 
where $\ve_n = (0, \ldots, 0, 1) \in \bfR^n$.

Next, let  
\[ R_0 = \sup \Big\{R: m_R \leq \frac{\tilde{m}}{2}\Big\}. \]
We may repeat the above argument with $\tF_{R_0}$, $m_{R_0}$, 
$c_{R_0}$ in place of $\tF$, $\tm$, $\tc$, respectively, to
derive an upper bound for $R_0$; we omit the details here. 
The proof of \eqref{hess-a10b} %in Theorem~\ref{3I-th4}
is therefore complete.

\bigskip

\section{Gradient estimates and proof of  
Lemma~\ref{ma-lemma-C10}}
%proof of Theorem~\ref{3I-thm2}}
\label{3I-G}
\setcounter{equation}{0}

\medskip

In this section we first derive the gradient estimate \eqref{hess-a10g}. 
Let $w = 1 + |\nabla u|^2$ and $\phi$
a function to be determined satisfying
\[ \phi  \geq 1, \;\; \phi' \leq - c_0, \;\; \phi'' \leq 0 \]
for some constant $c_0 > 0$. We consider 
%In our case we only need to take a very simple function:
%$\phi (u) = A - u$. 
 the function $\log {w} - \log {\phi (u)}$ which we assume to attain a maximum at an interior point $x_0 \in M$.
Using orthonormal local frames we have at $x_0$,
\begin{equation}
\label{3I-G05a}
   \frac{\nabla_j w}{w} - \frac{\phi' \nabla_j u}{\phi} = 0
\end{equation}
and 
 \begin{equation}
\label{3I-G15a}
\phi F^{ij} \nabla_{ij} w 
  - w F^{ij} (\phi' \nabla_{ij} u + \phi'' \nabla_{i} u\nabla_{j} u) \leq 0. 
\end{equation}
We calculate
%\begin{equation}
%\label{3I-G155a}
$\nabla_j w = 2 \nabla_k u \nabla_{jk} u
                     = 2 (U_{jk} - \chi_{jk}) \nabla_k u$
%\end{equation}
and by Schwarz inequality,
\begin{equation}
\label{3I-G25a}
\begin{aligned}
 F^{ij} \nabla_{ij} w 
      = \,& 2 F^{ij} \nabla_{ik} u \nabla_{jk} u 
               + 2 F^{ij} \nabla_{k} u \nabla_{ijk} u \\
 \geq \,& \frac{3}{2} F^{ij} U_{ik} U_{jk} 
                - C w \sum F^{ii} - 2 |\nabla \psi| \sqrt{w}
\end{aligned} 
\end{equation}

Assume that
$U_{ij} $ and $F^{ij}$ are diagonal at $x_0$. 
Let $I = \{i: n |\nabla_i u| \geq |\nabla u|\}$. 
We see that $I \neq \emptyset$
and by \eqref{3I-G05a} for $i \in I$, 
\[ U_{ii} = \frac{\phi' w}{2 \phi} 
     + \frac{\chi_{ki} \nabla_k u}{\nabla_i u} 
   \leq \frac{\phi' w}{2 \phi} + n |\chi| \equiv - K.  \]
We shall assume $K \geq - \frac{\phi' w}{4 \phi} > 0$; otherwise there is a
bound for $w$ and the proof is complete. 
Let $J = \{i: U_{ii} \leq - K\}$ so $I \subseteq J$. Clearly,
\begin{equation}
\label{3I-G35a}
 F^{ii} U_{ii}^2 \geq K^2 \sum_J F^{ii} 
                    \geq \frac{K^2}{n} \sum F^{ii}. 
\end{equation}

Let ${\mu} = \lambda (\nabla^2 \ul u (x_0)+  \chi (x_0))$,
$\lambda = \lambda (\nabla^2 u (x_0) + \chi (x_0))$ and
$\beta$ as in \eqref{gj-C145}.
Suppose first that $|\nu_{\mu}- \nu_{\lambda}| \geq \beta$.
By %\eqref{3I-R70}, \eqref{3I-R75} and 
Lemma~\ref{ma-lemma-C10} 
%\[ F^{ii} \nabla_{ii} v = F^{ii} \nabla_{ii} (\ul u - u)
 %   \geq \varepsilon \sum F^{ii} + \varepsilon \]
%for some $\varepsilon > 0$.
we obtain 
\[ - \sum_{U_{ii} < 0} F^{ii} U_{ii} 
      \geq \varepsilon - C \sum F^{ii} \]
and therefore, 
\begin{equation}
\label{3I-G45a}
F^{ii} U_{ii}^2 \geq - K \sum_J F^{ii} U_{ii}
\geq - \frac{K}{n} \sum_{U_{ii} < 0} F^{ii} U_{ii}  
\geq \frac{\varepsilon K}{n} -  C K \sum F^{ii}.
\end{equation}

Suppose now that 
$|\nu_{\mu}- \nu_{\lambda}| <  \beta$. 
By \eqref{hess-a300} % there are $R, b_0 > 0$ such that
 when $|\lambda|$ is large enough,
\[ |\lambda| \sum f_i (\lambda) \geq  b_0 \]
for some $b_0 > 0$.
It follows from \eqref{hess-a276} that when $K$
 is sufficiently large, 
\begin{equation}
\label{3I-G55a}
F^{ii} U_{ii}^2 
    \geq \frac{\beta |\lambda|^2}{\sqrt{n}} \sum F^{ii}
    \geq \frac{b_0 \beta |\lambda|}{\sqrt{n}}
   \geq \frac{b_0 \beta K}{\sqrt{n}}.
\end{equation}

Finally, taking
$\phi (u) = - u + 1 + \max_{\bM} u$
and combining \eqref{3I-G15a}-\eqref{3I-G55a}, 
we derive by condition~\eqref{3I-55a'}  a bound for 
$w (x_0)$ in either case;
%combined with \eqref{3I-G45a} or \eqref{3I-G55a};
this is the only place in the paper we need to assume
\eqref{3I-55a'}.
The proof of the gradient estimate \eqref{hess-a10g}
is complete. 

Note that $\ul u \leq u \leq h$ where $h$ satisfies
$\Delta h + \tr \chi = 0$ in $M$ and $\partial M$.
Consequently,
\[ \max_{\bM} u +  \max_{\partial M} |\nabla u| \leq C. \]
We have therefore derived \eqref{hess-a10C2} and the 
proof of Theorem~\ref{3I-thm1} is complete.

We now give a proof of 
Lemma~\ref{ma-lemma-C10} which is slightly more general than the original version in \cite{Guan13a}; see
also \cite{GJ}.

\begin{proof}[Proof of Lemma~\ref{ma-lemma-C10}]
Let $P$ be the hyperplane through $\mu$ parallel to 
$T_{\lambda} \partial \Gamma^{f (\lambda)}$, the tangent plane at $\lambda$ to $\partial \Gamma^{f (\lambda)}$.
Since
%\begin{equation}
%\label{gj-C155}
$0 < \nu_{{\mu}} \cdot \nu_{{\lambda}} \leq 1 - \beta^2/2$,
%\end{equation}
we can find constants $\theta, \, \delta > 0$ and
\[ \zeta \in P \cap \partial B_{\delta} (\mu), \;\;
\eta \in \partial \Gamma^{f (\mu)} \cap \partial B_{\delta} (\mu), 
\]
all depending only on $\beta$ and smoothness of $f$ near $\mu$, such that
\[ f (\zeta) \geq f (\mu) + \theta \]
and 
\[ \nu_{{\lambda}} \cdot (\eta- {\lambda}) 
\leq \nu_{{\lambda}} \cdot (\mu- {\lambda}) - \theta. \]
By the concavity of $f$,
\begin{equation}
\label{gj-C160}
\begin{aligned}
  \sum f_i ({\lambda}) ({\mu}_i - {\lambda}_i)
 = \,& \sum f_i ({\lambda}) (\zeta_i - {\lambda}_i) \\
  \geq \,& f (\zeta) - f ({\lambda}) 
  \geq f ({\mu}) - f ({\lambda}) + \theta
\end{aligned}
\end{equation}
and 
\begin{equation}
\label{gj-C161}
\begin{aligned}
  \sum f_i ({\lambda}) ({\mu}_i - {\lambda}_i)
\geq \,& \sum f_i ({\lambda}) (\eta_i - {\lambda}_i) 
              + \theta |Df (\lambda)| \\
\geq \,& f (\eta) - f ({\lambda}) 
              + \frac{\theta}{\sqrt{n}} \sum f_i \\
     = \,& f (\mu) - f ({\lambda}) 
              + \frac{\theta}{\sqrt{n}} \sum f_i. 
\end{aligned}
\end{equation}
We obtain \eqref{gj-C160'} for
$\varepsilon = \theta/\sqrt{n}$.
\end{proof}

\section{Appendix. Remarks on notions of extended 
subsolutions}
\label{3I-R}
\setcounter{equation}{0}

We end this revision by adding some remarks on results in \cite{Guan14}, \cite{Sze18} and the current paper. 
We first note that for equation~\eqref{3I-10} on
closed manifolds ($\partial M = \emptyset$) we can not use the subsolution assumption as it would mean that either it 
is alrady a solution or \eqref{3I-10} does not have any solution. This is a simple consequence of 
the maximum principle.  
In \cite{Guan14}  we introduced a condition and proved
a counterpart of Lemma~\ref{ma-lemma-C10} which enabled us to derive second order estimates on closed manifolds.

%we recall some facts from \cite{Guan14}.
For $\sigma \in (\sup_{\partial \Gamma} f,  \sup_{\Gamma} f)$, we recall by \eqref{3I-20} and \eqref{3I-30} that 
$\partial \Gamma^{\sigma}$,
the boundary of $\Gamma^{\sigma} = \{\lambda \in \Gamma: 
       f (\lambda) > \sigma\}$,
is a smooth and convex complete hypersurface in $\Gamma$.
Follwing \cite{Guan14} we define
\[ S^{\sigma}_{\mu} = \{\lambda \in \partial \Gamma^{\sigma}:
\nu_{\lambda} \cdot (\mu - \lambda) \leq 0\}, \;\;
\mu \in \Gamma \]
 and 
\[ \mathcal{C}_{\sigma}^+ = \{\mu \in \Gamma: S^{\sigma}_{\mu} \;
  \mbox{is compact}\}. \]
We call $\partial \mathcal{C}_{\sigma}^+$
{\em the tangent cone at infinity} of $\partial \Gamma^{\sigma}$.
Clearly $\mathcal{C}_{\sigma}^+$ is a convex symmetric cone
and contains $\Gamma$. 

In \cite{Guan14} we proved that $\mathcal{C}_{\sigma}^+$
is open and used the following condition
in deriving second order estimates: 
there exists an admissible function $\ul u \in C^2 (\bM)$ statisfying
\begin{equation}
\label{3I-3}
 \lambda (\nabla^2 \ul u + \chi) (x) \in  \mathcal{C}_{\psi (x)}^+,
%\partial \Gamma^{\sigma} \neq \mathcal{C}_{\sigma}
\;\; \forall \; x \in \bM.
\end{equation}
Later on Sz\'ekelyhidi~\cite{Sze18} treated the complex counterpart of equation~\eqref{3I-10} on closed Hermitian manifolds. Among many interesting results and ideas, he introduced a notion of generalized 
subsolutions, called $\cC$-subsolutions, which for
equation~\eqref{3I-10} can be expressed as 
\begin{equation}
\label{3I-3Sze}
 (\lambda (\nabla^2 \ul u + \chi) (x) + \Gamma_n) 
\cap \partial  \Gamma^{\psi (x)} 
\;\; \mbox{is compact}
\;\; \forall \; x \in \bM.
\end{equation}

It turns out that for type 1 cones the two notions are equivalent; here we shall only give a proof that \eqref{3I-3Sze} implies \eqref{3I-3}. For this it enough to show that if $\mu \in \Gamma$ satisfies 
\begin{equation}
\label{3I-3Sze'}
 (\mu + \Gamma_n) 
\cap \partial  \Gamma^{\sigma} 
\;\; \mbox{is compact}
\end{equation}
then $\mu \in \cC_{\sigma}^+$.
Recall from \cite{CNS3} that a cone $\Gamma$ is of type 1 if
all positive $\lambda_i$ axis belong to $\partial \Gamma$. 

%Part {\bf a)} follows from the following observation.
Let $T$ be a supporting plane to $\cC_{\sigma}^+$
with unit normal vector $\eta = (\eta_1, \ldots, \eta_n)$. We may assume $\eta_n \geq \cdots \geq \eta_1 \geq 0$ since
$\eta \in \ol\Gamma_n$.
As $\Gamma$ is of type 1, through each axis there are supporting planes to $\Gamma$.
Let $P$ be a supporting plane to $\Gamma$ through
the $\lambda_1$-axis, and 
$\nu$ denote the unit normal vector to $P$.
Then $\nu_1 = \nu \cdot e_1 = 0$ where $e_i$ denotes the unit vector in the positive 
$\lambda_i$-axis direction, and 
\[ h := \inf_{\lambda \in \Gamma^{\sigma}} 
          \nu \cdot \lambda \geq 0 \]
since $\Gamma^{\sigma}$ is contained in $\Gamma$. 
We see that $\eta_1 = 0$. 

Suppose now that $\mu \in \Gamma$ satisfies \eqref{3I-3Sze'}  but $\mu \notin \cC_{\sigma}^+$. 
Let
$\mu' \in \partial \cC_{\sigma}^+$ with 
\[ |\mu - \mu'| = 
   \mbox{dist} (\mu, \tilde{\cC}_{\sigma}^+) > 0. \]
Clearly 
$\nu' := (\mu' - \mu)/|\mu' - \mu'| \in \partial \Gamma_n$ as it is the unit normal to a supporting plane $P$ of $\cC_{\sigma}^+$ at $\mu'$. 
We may assume $\nu' = (0, \nu'_2, \ldots, \nu'_n)$
and let $\mu^{R} := \mu + (R, 0, \ldots, 0)$. Then
\[ \mbox{dist} (\mu^R, \cC_{\sigma}^+) 
   \geq \mbox{dist} (\mu^R, P) 
  = \nu' \cdot (\mu' - \mu^R) = |\mu' - \mu| > 0. \]
This contradicts to the assumption that 
$\mu^R \in \Gamma^{\sigma}$ when $R$ is large enough.

%\bigskip

\end{document}